\documentclass{amsart}

\usepackage{chngcntr}

\usepackage[utf8]{inputenc}

\usepackage{amsmath, amsthm, amssymb, amsfonts}
\usepackage{mathrsfs}
\usepackage{mathtools}
\usepackage[all]{xy}
\usepackage{nicefrac}
\usepackage{enumerate}
\usepackage{tikz-cd}

\usepackage{color}

\usepackage{booktabs}
\usepackage{graphicx}
\usepackage{varwidth}

\newcommand{\NN}{\mathbb{N}}
\newcommand{\ZZ}{\mathbb{Z}}
\newcommand{\QQ}{\mathbb{Q}}

\newcommand{\OO}{\mathscr{O}}
\newcommand{\LL}{\mathscr{L}}

\newcommand{\FF}{\mathscr{F}}

\newcommand{\TT}{\mathscr{T}}

\newcommand{\mm}{\mathfrak{m}}

\newcommand{\Aut}{\operatorname{Aut}}

\newcommand{\Gal}{\operatorname{Gal}}

\newcommand{\Hom}{\operatorname{Hom}}

\newcommand{\sep}{\operatorname{sep}}

\newcommand{\Spec}{\operatorname{Spec}}

\newcommand*{\longhookrightarrow}{\ensuremath{\lhook\joinrel\relbar\joinrel\rightarrow}}

\theoremstyle{plain}
\newtheorem{thm}{Theorem}[section]	
\newtheorem{prop}[thm]{Proposition}
\newtheorem{lem}[thm]{Lemma}

\newtheorem*{thm*}{Theorem}

\theoremstyle{definition}

\newtheorem{rem}[thm]{Remark}
\newtheorem{ex}[thm]{Example}
\newtheorem*{nota*}{Notation}

\theoremstyle{remark}
\newtheorem*{bem*}{Bemerkung}

\usepackage{hyperref}
\begin{document}
	
	\title{The torsion in the cohomology of wild elliptic fibers}
	
	\author{Leif Zimmermann}
	\address{Mathematisches Institut, Heinrich-Heine-Universität, Düsseldorf, 40204 Düsseldorf, Germany}
	\email{zimmermann@math.uni-duesseldorf.de}
	
	
	
	\date{\today}
	
	
	
	\begin{abstract}
		Given an elliptic fibration $f \colon X \to S$ over the spectrum of a complete discrete valuation ring with algebraically closed residue field, we use a Hochschild--Serre spectral sequence to express the torsion in $R^1f_\ast \OO_X$ as the first group cohomology $H^1(G,H^0(S^\prime, \OO_{S^\prime}))$. Here, $G$ is the Galois group of the maximal extension $K^\prime / K$ such that the normalization of $X \times_S S^\prime$ induces an \'etale covering of $X$, where $S^\prime$ is the normalization of $S$ in $K^\prime$. The case where $S$ is a Dedekind scheme is easily reduced to the local case. Moreover, we generalize to higher-dimensional fibrations.
	\end{abstract}
	
	\maketitle
	\tableofcontents

\section*{Introduction}

Elliptic fibrations play an eminent role in the classification of algebraic surfaces. They are defined as regular integral surfaces $X$ together with a surjective, proper morphism $f$ to a regular curve $S$ such that $f_\ast \OO_X = \OO_S$ holds and the generic fiber is a smooth curve of genus one.
If $f \colon X \to S$ is an elliptic fibration, we can write $R^1 f_\ast \OO_X$ as a sum of an invertible sheaf $\LL$ and a torsion module $\TT$. The latter one is only non-trivial at finitely many points. In characteristic zero, it even vanishes completely. Yet, in positive characteristic, this is not the case in general. Fibers over closed points $s \in S$ such that $\TT_s$ is non-trivial are called \emph{wild fibers}. They are obstructions for the cohomological flatness of $f$. We examine the structure of this torsion module $\TT$ in this work as follows:

First, we introduce the notation in the first subsection of Section~\ref{Sec:Basics} and argue that we may localize at some closed point $s \in S$ and complete at its stalk. Hence, we may assume $S$ to be the spectrum of a complete discrete valuation ring. Furthermore, we show that $H^1(X,\OO_X)$ depends only on the generic fiber of an elliptic fibration $f \colon X \to S$. These results also hold for a natural generalization of the notion of elliptic fibration we call \emph{abelian fibration}, that is, a fibration with regular total space such that the generic fiber is a torsor under an abelian variety. In the second subsection, we recall the well-known possible reduction types an elliptic fibration or more generally the N\'eron model of an abelian variety may have. Afterwards, we introduce the main technique we are going to use in the third subsection. Before we start studying the first cohomlogy group of an elliptic fibration in Section~\ref{Sec:ApplGoodMultReduction}, we need to study maximal \'etale coverings of a fibration in Section~\ref{Sec:MaximalEtaleCovers}. Therein, we recall the notion of the maximal base field of \'etale coverings of the total space from \cite{Mitsui2015}. It is obtained by composing all finite field extensions $K^\prime$ of the function field $K$ of $S$ with the following property: If $S^\prime$ denotes the normalization of $S$ in $K^\prime$, the map $X^\prime \to X \times_S S^\prime \to X$ is \'etale, where $X^\prime \to X \times_S S^\prime$ is the normalization. As a first result, we show that the base field of \'etale coverings only depends on the generic fiber for abelian fibrations, and obtain:

\begin{thm*}
(See \ref{Thm:TorsionNormalFree}) Let $f \colon X \to S$ be an abelian fibration with $S$ being the spectrum of a complete discrete valuation ring $R$ with algebraically closed residue field. Denote by $M = M_{X/K}$ the maximal base field of \'etale coverings of the total space and by $R^\prime$ the integral closure of $R$ in $M$. Assume that $X^\prime$, the normalization of the base change of $X$ to $R^\prime$, has free $H^1(X^\prime, \OO_{X^\prime})$.
Then the torsion in $H^1(X,\OO_X)$ is canonically isomorphic to $H^1(G, R^\prime)$, where $G$ is the Galois group of $M/K$.
\end{thm*}

Now most of Section~\ref{Sec:ApplGoodMultReduction} is devoted to the question in which cases the cohomology group $H^1(X^\prime, \OO_{X^\prime})$ is free. With notation as in the above theorem, this will be the case under the following assumptions (see Theorem~\ref{Thm:TorsionGoodCase} and Theorem~\ref{Thm:TorsionMultCase}):
\begin{enumerate}
	\item $A_K$ is an abelian variety with good reduction and $X_K(M) \neq \varnothing$,
	\item $A_K$ is an elliptic curve with ordinary good reduction,
	\item $A_K$ is an elliptic curve with multiplicative reduction.
\end{enumerate}

Now assume that $A_K$ does not have semi-stable reduction. We can potentially reduce the situation to the above cases by applying the theory of semi-stable reduction: One always finds an extension $K^\prime / K$ such that $A_{K^\prime}$ has semi-stable reduction. If $S$ is the spectrum of a complete discrete valuation ring with separably closed residue field, there is a minimal extension among all extensions $K^\prime$ such that $A_{K^\prime}$ becomes semi-stable, and its degree is prime to the residue characteristic $p$ for $p > 2g+1$, where $g = \dim A_K$, cf. \cite{Conrad:SemStabRedForAbelVar}, Theorem~6.8. We use this in Section~\ref{Sec:AddReduction} to obtain:

\begin{thm*}
(See \ref{Thm:AdditiveReduction}) Let $A_K$ be an abelian variety that attains semi-stable reduction over a Galois extension $K^\prime/K$, with Galois group $G$ of order prime to the residue characteristic $p$. Assume that either $A_{K^\prime}$ has good reduction or that $A_K$ is an elliptic curve. If $X$ is an abelian fibration over $S$ such that the generic fiber $X_K$ is a torsor under $A_K$, then
\[
H^1(X, \OO_X) = H^1(X^\prime, \OO_{X^\prime})^G,
\]
where $X^\prime$ is any regular proper model of $X_{K^\prime}$ over the integral closure $S^\prime$ of $S$ in $K^\prime$.
\end{thm*}

Note that a regular model $X^\prime$ of $X_{K^\prime}$ as mentioned in the above theorem exists in the case of good reduction for arbitrary dimensions due to \cite{Liu+Lorenzini+Raynaud}, Proposition~8.1 resp.\ \cite{Lewin-Menegaux:ModelesMinimaux}.
The groups $H^1(G, R^\prime)$ which we identify with the torsion in $H^1(X, \OO_X)$ are quite accessible in terms of the jumps of higher ramification groups: In the last section, we apply a formula by Sen \cite{Sen:AutoLocFields} for $H^1(G,R^\prime)$ in the case of elliptic fibrations where the residue field is algebraically closed.\\

\textbf{Acknowledgment.} This work is part of my PhD thesis and I wish to express my gratitude to my advisor Stefan Schröer. For financial support from the project ``SCHR 671/6-1 Enriques-Mannigfaltigkeiten'', I would like to thank the Deutsche Forschungsgemeinschaft. It also funds the research training group \emph{GRK~2240: Algebro-geometric Methods in Algebra, Arithmetic and Topology}, in whose framework the research was conducted. Furthermore, I am indebted to the anonymous referee for greatly improving the readability of this article.

\section{Basics and reduction to the local case} \label{Sec:Basics}

The main goal of this article is to study the torsion in $R^1 f_\ast \OO_X$ for an elliptic fibration $f \colon X \to S$. In the first subsection, we will start by defining fibrations over Dedekind schemes, see how the torsion in the higher direct images arises and that the study of its structure is essentially local. Moreover, under the condition of $X$ being regular, we will see that the higher direct images only depend on the generic fiber. Finally, we fix the notation.

In the next subsection, we will come closer to our prior objects of interest - elliptic and abelian fibrations - and recapitulate the theory of models and reduction type of abelian varieties.

The last subsection introduces the main technique for getting insight into the torsion structure.

\subsection{Reduction to the local case}

We will denote by $S$ a \emph{Dedekind scheme}, i.e.\ a connected, noetherian, normal scheme of dimension one, with generic point $\eta$ and function field $K$. Given a normal, integral scheme $X$, we call a morphism $f \colon X \to S$ a \emph{fibration} if $f$ is proper and $\OO_S = f_\ast \OO_X$ holds. Because $S$ is a Dedekind scheme, we may decompose the higher direct images $R^i f_\ast \OO_X$ into the sum of a locally free sheaf $\LL_i$ of rank $n_i = \dim_K H^i(X_\eta, \OO_{X_\eta})$ and some torsion sheaf $\TT_i$ supported at finitely many points: The sheaf $R^i f_\ast \OO_X$ is coherent due to the properness of $f$. 
As cohomology commutes with flat base change, there is a canonical isomorphism $(R^i f_\ast \OO_X)_\eta \to H^i(X_\eta, \OO_{X_\eta}) \simeq \OO_{S,\eta}^{\oplus n_i}$ of $\OO_{S,\eta}$-vector spaces. This isomorphism extends to an open subset of $S$, 
which is also dense due to the irreducibility of $S$. Thus, $R^i f_\ast \OO_X$ is free of rank $n_i$ except at finitely many closed points. After localizing with respect to such a point, $R^i f_\ast \OO_X$ becomes a finitely generated module over a local Dedekind ring, i.e.\ a discrete valuation ring. By the structure theorem for finitely generated modules over discrete valuation rings, it must be isomorphic to the sum of a free part and a torsion part.

So to understand $\TT_i$, it suffices to understand $(\TT_i)_s$ over closed points $s \in S$. To study this module, we will base change by $\Spec(\OO_{S,s})$. As the invariant factors of $(\TT_i)_s$ do not change when completing the stalk at $s$, we will even restrict our problem to the case where $S$ is the spectrum of a complete discrete valuation ring. This is fine as the base change $X \times_S \Spec(\hat{\OO}_{S,s}) \to \Spec(\hat{\OO}_{S,s})$ is again a fibration.

So let $S$ be the spectrum of a complete discrete valuation ring. As a further reduction step, we consider fibrations $f \colon X \to S$ where $X$ is assumed to have \emph{pseudo-rational singularities}, i.e.\ we additionally impose that $X$ is Cohen-Macaulay, admits a dualizing complex and that for every normal scheme $Y$ and projective, birational morphism $g \colon Y \to X$, the map between dualizing complexes $\varsigma_g \colon g_\ast \omega_Y \to \omega_X$ is an isomorphism (see \cite{Kovacs:RationalSingularities}, Definition~9.4). Regular schemes have pseudo-rational singularities. We say that a fibration $f \colon X \to S$ has pseudo-rational singularities if $X$ has.

\begin{prop} \label{Prop:InvariancePseudoRational}
	Let $f \colon X \to S$ be a fibration with pseudo-rational singularities. Then the cohomology groups $H^i(X, \OO_X)$ for $i \ge 0$ depend only on the generic fiber, that is, if $g \colon Y \to S$ is another fibration with pseudo-rational singularities such that $Y_\eta \simeq X_\eta$, then $H^i(X,\OO_X) \simeq H^i(Y, \OO_Y)$ for all $i \ge 0$.
\end{prop}
\begin{proof}
	Given another fibration $g \colon Y \to S$ with pseudo-rational singularities and isomorphism $\iota \colon X_\eta \to Y_\eta$, we take the graph $\Gamma$ of $\iota$ and its closure $Z$ in $X \times_S Y$ with reduced structure sheaf. The restriction of the projections to $Z$ gives proper birational morphisms $Z \to X$ and $Z \to Y$, i.e.\ $X$ and $Y$ are \emph{properly birational}. Hence, we may apply \cite{Kovacs:RationalSingularities}, Theorem~1.6, to see that $R^if_\ast \OO_X  \simeq R^i g_\ast \OO_Y$ holds. As $S$ is affine, we have $R^if_\ast \OO_X = H^i(X, \OO_X)$ and the same for $Y$, proving the claim.
\end{proof}

Therefore, if $X$ is regular like in the case of elliptic fibrations, it is natural to ask for a description of the torsion in $H^1(X,\OO_X)$ in terms of the generic fiber. To ease notation, we will from now on \emph{always assume that $S$ is the spectrum of a complete discrete valuation ring~$R$}, with generic point $\eta$, function field $K$, closed point $s$ and \emph{separably closed} residue field $k$ if not explicitly stated otherwise.

\subsection{Elliptic and Abelian fibrations}

An \emph{elliptic fibration} $f \colon X \to S$ is a fibration with $X$ regular such that the generic fiber $X_\eta$ is a smooth curve of genus one. The smoothness condition on $X_\eta$ is equivalent to the fact that $X_\eta$ is a torsor under an elliptic curve, namely its Jacobian.
We therefore generalize this notion and call a fibration $f \colon X \to S$ with $X$ regular an \emph{abelian fibration} if $X_\eta$ is a torsor under an abelian variety.

Recall that for an abelian variety $A_K$ over $K$, an $A_K$-torsor is a separated scheme $X_K$ of finite type over $K$ equipped with an $A_K$-action such that there exists an equivariant isomorphism to $A_K$ after some field extension. If such an isomorphism exists over $K^\prime$, we say that $X_K$ \emph{splits} over $K^\prime$.

According to Proposition~\ref{Prop:InvariancePseudoRational}, the torsion in $H^1(X,\OO_X)$ of an abelian fibration $f \colon X \to S$ only depends on the generic fiber. Hence, it is natural to introduce the notion of a model of an abelian variety. In particular, the N\'eron model will be important to us, as we will make a case-by-case analysis according to the closed fiber of the N\'eron model.

Given a proper smooth connected scheme $X_K$ over a field $K$, we call a scheme $X$ over $S$ a \emph{model of $X_K$ over $S$} if $X$ is normal and flat, of finite type over $S$ together with an isomorphism $X_\eta \simeq X_K$ which we will often omit. Fibrations are models of their generic fiber. A \emph{N\'{e}ron model of $X_K$ over $S$} is a model $N$ smooth and separated over $S$ satisfying the following universal property:
\begin{center}
	For each smooth $S$-scheme $Y$ and each $K$-morphism $u_\eta \colon Y_\eta \to N_\eta$,\\
	there is a unique $S$-morphism $u \colon Y \to N$ extending $u_\eta$.
\end{center}
The latter property is also called the \emph{N\'{e}ron mapping property}, which makes the N\'eron model unique up to unique isomorphism. The N\'{e}ron model $N$ of an abelian variety always exists (cf.\ \cite{BLR:NeronModels}, Proposition 1.2/8), and in that case, it inherits a group scheme structure via the N\'{e}ron mapping property. Therefore, the closed fiber over the closed point $s$ is again a smooth separated group scheme. By Chevalley's Theorem (e.g.\ \cite{BLR:NeronModels}, Theorem 9.2/1), the identity component of the geometric fiber at the closed point $s$ is uniquely an extension of an abelian variety by a connected affine group, which in turn splits as the product of a torus and a unipotent group. In fact, if the unipotent part is trivial, the extension is already defined over $k$, see \cite{BLR:NeronModels}, paragraph below Lemma 7.3/1. We therefore say that the abelian variety $A_K$ with N\'eron model $N$ over $S$ has \emph{good reduction} or \emph{abelian reduction at $s$} if the identity component $N_{s}^0$ of $N_s$ is an abelian variety and that $A_K$ has \emph{semi-stable reduction at $s$} if $N_{s}^0$ is the extension of an abelian variety by a torus.

Note that in the case of elliptic curves, the N\'eron model can be identified with the smooth locus of a minimal regular elliptic fibration, cf.\ \cite{Liu}, Theorem~10.2.14. In the case of elliptic curves, if $N_s^0$ is a torus, we say that $A_K$ has \emph{multiplicative reduction}. If $N_s^0$ is unipotent, we say that $A_K$ has \emph{additive reduction}.

\subsection{Introducing the main technique}

In this subsection, we are going to prove the key observation of this article. We call $q \colon X^\prime \to X$ a \emph{finite \'etale Galois covering} if it is finite, \'etale and the number of elements in $G = \Aut(X^\prime / X)$ equals the degree of $q$.

\begin{prop} \label{Prop:TorsionInEtaleCase}
Let $f \colon X \to S$ be a fibration and $q \colon X^\prime \to X$ be a finite \'etale Galois covering with Galois group $G = \Aut(X^\prime / X)$. If $\FF$ is a coherent sheaf on $X$ such that $H^1(X^\prime, q^\ast \FF)$ is free, then
\[
	H^1(X,\FF) = R^{\oplus d} \oplus H^1(G, H^0(X^\prime, q^\ast \FF))
\]
holds for $d = \operatorname{rank} H^1(X, \FF)$.
\end{prop}
\begin{proof}
Recall that for any abelian sheaf $F$ on $X$ for the \'etale topology, one has the \emph{Hochschild--Serre spectral sequence}
\[
H^r(G,H^s(X_{\text{\'et}}^\prime,q^\ast F)) \Longrightarrow H^{r+s}(X_{\text{\'et}},F),
\]
see \cite{Milne:EtaleCohomology}, Chapter~III, Theorem 2.20. In particular, if $F = \FF$ is a quasi-coherent Zariski sheaf, like the structure sheaf $\OO_X$, the \'etale and Zariski cohomology coincide, cf.\ \cite{Milne:EtaleCohomology}, Remark~3.8. Hence, taking the five-term exact sequence yields
\[
0 \longrightarrow H^1(G, H^0(X^\prime, q^\ast \FF)) \longrightarrow H^1(X, \FF) \longrightarrow H^1(X^\prime, q^\ast \FF)^G.
\]
If $\FF$ is a coherent sheaf, we may write $H^1(X, \FF) = R^{\oplus d} \oplus \TT$ for some $d \ge 0$, where $\TT$ is the torsion in $H^1(X,\FF)$. If we assume $H^1(X^\prime, q^\ast \FF)$ to be free, we obtain $\TT = H^1(G, H^0(X^\prime, q^\ast \FF))$: The $R$-module $H^1(G, H^0(X^\prime, q^\ast \FF))$ is torsion because $G$ is a finite group. Therefore, its image lies in $\TT$. On the other hand, as $H^1(X^\prime, q^\ast \FF)$ is free by assumption, every element of $\TT$ is mapped to zero, hence lies in the image of $H^1(G, H^0(X^\prime, q^\ast \FF))$
\end{proof}

The following example is a localization of the construction done in \cite{Katsura+Ueno:EllipticSurfacesInCharacteristicP}, Example~8.1, where Katsura and Ueno calculate examples of wild elliptic fibers and determine the length of the torsion in $H^1(X,\OO_X)$.

\begin{ex} \label{Ex:ExampleKU}
Let $k$ be a field of characteristic $p > 0$ and $R = k[[t]]$ with field of fractions $K = k((t))$. Take an Artin--Schreier extension $K^\prime = K[x]/(x^p-x-t^m)$ for some $m$ prime to $p$, with integral closure $R^\prime$ of $R$ in $K^\prime$. Given an ordinary elliptic curve $E_k$ over $k$ with non-trivial $p$-torsion point $P$, define $X^\prime = E_k \otimes_k R^\prime$. This gives a fibration over $S^\prime = \Spec(R^\prime)$ and $G = \Gal(K^\prime / K)$ acts on $X^\prime$ via translation on the first factor and Galois action on the second factor. The quotient morphism $q \colon X^\prime \to X = X^\prime / G$ is \'etale as the group acts freely on the closed points, cf.\ \cite{Mumford:AbelianVarieties},
Theorem in Chapter II, Section 7. Hence we may apply Proposition~\ref{Prop:TorsionInEtaleCase} and obtain $H^1(X,\OO_X) = R \oplus H^1(G,R^\prime)$.
\end{ex}

\section{Maximal \'{e}tale covers} \label{Sec:MaximalEtaleCovers}

Given a fibration $f \colon X \to S$, we have seen in Proposition~\ref{Prop:TorsionInEtaleCase} that the \'etale coverings play an important role in determining the torsion of $H^1(X,\OO_X)$. To every finite \'{e}tale morphism $Y \to X$, we can form the Stein factorization $Y \to T \to S$ of $Y \to S$. We denote by $M = M_{X/S}$ the composite of all finite extensions $K(T)$ obtained in this way in some separable closure. If additionally $K(X) / K$ is separable (e.g.\ for abelian fibrations), it is called the \emph{maximal base field (of \'{e}tale coverings of the total space)} in \cite{Mitsui2015}, Definition~4.9. Write the special fiber $X_k$ as the sum of prime divisors $\sum_{i=1}^r m_i X_i$ and denote by $n_i$ the degree of the purely inseparable field extension $k_i / k$, where $k_i$ is the integral closure of $k$ in $\Gamma(X_i, \OO_{X_i})$. We have the following properties (cf.\ \cite{Mitsui2015}, Proposition~4.11, Lemma~4.12, Lemma~4.14 and Proposition~4.15):

\begin{prop} \label{Prop:EtaleBaseField}
Assume that the extension $K(X) / K$ is separable and denote $m = \gcd_i (m_i)$ as well as $n = \gcd_i (n_i)$. Then:
\begin{enumerate}
\item $M$ is finite and Galois, with $[M \colon K] \mid mn$.
\item Any subextension $K^\prime$ of $M/K$ induces an \'etale covering of $X$. More precisely, if $S^\prime$ is the normalization of $S$ in $K^\prime$, the normalization of $X \times_S S^\prime$ composed with the projection on $X$ is \'etale.
\item For every finite extension $K^\prime / K$ with $X(K^\prime) \neq \varnothing$, we have $M \subset K^\prime$.
\item If $g \colon Y \to X$ is a proper birational $S$-morphism between two fibrations over $S$ with $X$ regular, then $M_{X/S} = M_{Y/S}$.
\end{enumerate}
\end{prop}

Note that the last statement is stated in \cite{Mitsui2015}, Proposition~4.11, under the additional assumption that $Y$ is regular as well, but the proof only needs $Y$ to be normal. This is important for us, as we can now prove:

\begin{prop} \label{Prop:EtaleFieldDependsGenericFiber}
The maximal base field $M_{X/S}$ of a fibration $f \colon X \to S$ with $X$ regular and $K(X)/K$ separable only depends on the generic fiber $X_\eta$, i.e.\ for every other fibration $g \colon Y \to S$ with $Y_\eta \simeq X_\eta$, we have $M_{Y/K} = M_{X/K}$
\end{prop}
\begin{proof}
Like in the proof of Proposition~\ref{Prop:InvariancePseudoRational}, we take the graph of an isomorphism of the generic fibers, its closure in $X \times_S Y$ with reduced structure sheaf and normalize it. As $S$ is complete, all schemes considered are excellent and therefore the normalization here is a finite morphism of degree one. So the construction yields a fibration $Z \to S$ and proper $S$-morphisms $Z \to X$, $Z \to Y$ which are isomorphisms on the generic fiber. Applying Proposition~\ref{Prop:EtaleBaseField} $(iv)$ gives $M_{Y/S} = M_{Z/S} = M_{X/S}$.
\end{proof}

Because the residue field $k$ is assumed to be separably closed, there exists a unique extension $K^\prime / K$ of degree $d$ prime to the residue characteristic $p$ by adjoining a $d$-th root of unity. It is Galois. This leads to:

\begin{prop} \label{Prop:PrimeToEtale}
Assume $X$ to be regular and let $d$ be the prime-to-$p$ part of $m = \gcd_i (m_i)$. Then $d \mid [M \colon K]$.
\end{prop}
\begin{proof}
	Let $K^\prime / K$ be the unique extension of degree $d$ given by $K^\prime= K[T]/(T^d-u)$, where $u \in R$ is a uniformizer. Set $R^\prime$ to be the integral closure of $R$ in $K^\prime$ with spectrum $S^\prime$. We have to show that the normalization of $X \times_S S^\prime$ gives an \'etale covering $\pi \colon X^\prime \to X$.
	By Zariski--Nagata purity (cf.\ \cite{SGA_1}, Chapter~X, Th\'eor\`eme~3.1), it suffices to check \'etaleness on the points of codimension one. So let $P \in X$ be a point of codimension one and $A = \OO_{X,P}$ its stalk, a discrete valuation ring due to $X$ being regular. The base change $A[T]/(T^d-u)$ of $A$ by the flat morphism $R \to R^\prime = R[T]/(T^d-u)$ is again flat. If $P$ lies over $\eta$, the uniformizer $u$ is a unit in $A$ and $\Spec(A[T]/(T^d-u)) \to \Spec(A)$ is unramified because the fiber over the residue field $\kappa(A)$ of $A$ is given by the finite separable field extension $\kappa(A)[T]/(T^d-u)$. Hence, $\Spec(A[T]/(T^d-u)) \to \Spec(A)$ is \'etale and $A[T]/(T^d-u)$ is already normal.
	
	Now if $P$ lies over the closed point $s \in S$ with residue field $k$, the base change $A[T]/(T^d-u)$ is not normal anymore: As $A$ is a discrete valuation ring, we can write $u = a t^{ml}$ in $A$, where $a$ is a unit, $t$ a uniformizer in $A$ and $l$ a positive integer. Denoting the residue class of $T$ in $A[T]/(T^d-u)$ by $\sqrt[d]{u}$, the element $\sqrt[d]{a} = \sqrt[d]{u} \ t^{-(md^{-1})l}$ is a $d$-th root of $a$ and the field of fractions $L^\prime$ of $A[T]/(T^d-u)$ is isomorphic to both $L(\sqrt[d]{u})$ and $L(\sqrt[d]{a})$, where $L = K(X)$. In particular, $\sqrt[d]{a}$ satisfies an integral equation, but $\sqrt[d]{a}$ is not element of $A[T]/(T^d-u)$.
	On the other hand, consider $A[T]/(T^d-a)$. Its field of fractions is $L^\prime$ and the homomorphism
	\[
	A[T]/(T^d-u) \longrightarrow A[T]/(T^d-a), \quad T \longmapsto t^{\frac{ml}{d}} T
	\]
	is the normalization of $A[T]/(T^d-u)$ in $L^\prime$. Like in the case where $P$ lies over $\eta$, as $a$ is a unit in $A$ and $p$ is coprime to $d$, the map $\Spec(A[T]/(T^d-a)) \to \Spec(A)$ is \'etale and we are done.
\end{proof}

Obviously, the \'{e}tale Galois covering $X^\prime \to X$ associated to $M$ is a good candidate to apply Proposition~\ref{Prop:TorsionInEtaleCase} to: 

\begin{thm} \label{Thm:TorsionNormalFree}
Let $f \colon X \to S$ be a fibration such that the extension $K(X)/K$ is separable (e.g.\ $f$ is an abelian fibration). Denote by $M = M_{X/K}$ the maximal base field of \'etale coverings of the total space and by $R^\prime$ the integral closure of $R$ in $M$. Assume that $X^\prime$, the normalization of the base change of $X$ to $R^\prime$, has free $H^1(X^\prime, \OO_{X^\prime})$.
Then the torsion in $H^1(X,\OO_X)$ is canonically isomorphic to $H^1(G, R^\prime)$, where $G$ is the Galois group of $M/K$.
\end{thm}

So to study the torsion in the first cohomology group of abelian fibrations, we need to check whether the first cohomology group of the normalization of the base change of $X$ to $R^\prime$ is free.

\section{Application to good and multiplicative reduction} \label{Sec:ApplGoodMultReduction}

Let $f \colon X \to S$ be an abelian fibration. For the rest of this section, denote by $M = M_{X/K}$ the maximal base field of \'etale coverings of the total space and by $R^\prime$ the integral closure of $R$ in $M$, with spectrum $S^\prime$. The normalization of the base change of $X$ to $R^\prime$ will be denoted by $X^\prime$. In this section, we will take advantage of Theorem~\ref{Thm:TorsionNormalFree} and study the question when $H^1(X^\prime, \OO_{X^\prime})$ is free.

\subsection{The case of multiplicative reduction}

Assume that $f \colon X \to S$ is a minimal elliptic fibration such that the Jacobian of $X_K$ has multiplicative reduction. By \cite{Liu+Lorenzini+Raynaud}, Proposition~8.3, there is a unique cyclic extension $K^\prime / K$ minimal among all other extensions $L / K$ such that $X_K(L) \neq \varnothing$. This result lies at the core of the following theorem:

\begin{thm}\label{Thm:TorsionMultCase}
Let $f \colon X \to S$ be an elliptic fibration such that the Jacobian of $X_K$ has multiplicative reduction. Then $H^1(X^{\prime},\OO_{X^\prime})$ is free.
\end{thm}
\begin{proof}
Let $f \colon X \to S$ be an elliptic fibration. Using Proposition~\ref{Prop:EtaleFieldDependsGenericFiber} and Proposition~\ref{Prop:InvariancePseudoRational}, $M$ and the cohomology groups of $X^\prime$ only depend on the generic fiber. We therefore may assume that $f \colon X \to S$ is a minimal fibration. The proof of \cite{Liu+Lorenzini+Raynaud}, Proposition~8.3, then shows that $M = M_{X/S}$ from Section~\ref{Sec:MaximalEtaleCovers} is equal to $K^\prime$, the unique cyclic extension minimal among all other extensions $L / K$ such that $X_K(L) \neq \varnothing$. Hence, $X_M$ has a rational point and $X^\prime$ a section. Therefore, the multiplicity of its closed fiber is one and it is cohomologically flat by \cite{Raynaud1970}, Th\'eor\`eme~8.2.1. Thus, by \cite{CossecDolgachev:EnriquesSurfaces}, p.\ 287, the $R$-module $H^1(X^\prime, \OO_{X^\prime})$ is free.
\end{proof}

\begin{rem}
Let $X_K$ have additive reduction. Then there is no \'{e}tale covering of $X$ by \cite{Mitsui2015}, Case (A) on p.\ 1121 (the argument also works for $k$ separably closed), so $M = K$. Note that in particular, the multiplicity $m = \gcd_i (m_i)$ of the closed fiber is a power of $p$ by Proposition~\ref{Prop:PrimeToEtale}.
\end{rem}

\subsection{The case of good reduction}

This subsection is dedicated to the proof of the following:

\begin{thm} \label{Thm:TorsionGoodCase} 
Let $f \colon X \to S$ be an abelian fibration, and assume that the generic fiber $X_K$ is a torsor under an abelian variety $A_K$ with good reduction. Assume that one of the following conditions hold:
\begin{enumerate}
\item $X_K(M)$ is non-empty,
\item $k$ is algebraically closed and $A_K$ is an elliptic curve with ordinary good reduction.
\end{enumerate}
Then $H^1(X^\prime,\OO_{X^\prime})$ is free.
\end{thm}
\begin{proof}[Proof in case $X_K(M) \neq \varnothing$]
Denote by $A$ the N\'eron model of $A_K$ over $S$. As $A_K$ has good reduction, $A$ is an abelian scheme, so its base change $A \otimes_R R^\prime$ is an abelian scheme as well and therefore has free cohomology groups. Now if $X_K(M)$ is non-empty, the generic fibers $X_M$ and $A_M$ of $X^\prime$ and $A \otimes_R R^\prime$ are isomorphic. The claim now follows from Proposition~\ref{Prop:InvariancePseudoRational}.
\end{proof}

Before proving Theorem~\ref{Thm:TorsionGoodCase} in the remaining case, we will need some preparatory results. These also hold for abelian fibrations such that the generic fiber is a torsor under an abelian variety with good reduction. Therefore, we will stick with this setting until we actually prove Theorem~\ref{Thm:TorsionGoodCase} under assumption $(ii)$.

Due to Proposition~\ref{Prop:EtaleFieldDependsGenericFiber} and Proposition~\ref{Prop:InvariancePseudoRational}, we may restrict ourselves to any abelian fibration that has generic fiber isomorphic to $X_K$. By \cite{Liu+Lorenzini+Raynaud}, Proposition~8.1 resp.\ \cite{Lewin-Menegaux:ModelesMinimaux}, there is a canonical choice of a regular fibration $X$ over $S$ with generic fiber isomorphic to $X_K$: It has the property that the action of $A_K$ on $X_K$ extends to an action $A \times_S X \to X$, where $A$ is the N\'{e}ron model of $A_K$ over $S$. The induced map $A \times_S X \to X \times_S X$, $(a,x) \mapsto (ax,x)$ is surjective. Let $L / K$ be a finite Galois extension such that $X_K(L) \neq \varnothing$. We denote by $R^{\prime \prime}$ the integral closure of $R$ in $L$ and $S^{\prime \prime}$ its spectrum. The base change $A_{S^{\prime \prime}}$ of $A$ to $S^{\prime \prime}$ is the N\'{e}ron model of $A_{L}$ over $S^{\prime \prime}$. Then $X$ is the quotient of $A_{S^{\prime \prime}}$ by a finite and flat equivalence relation by its construction in loc.cit. Note that due to $X_K(L) \neq \varnothing$ and Proposition~\ref{Prop:EtaleBaseField}, $M = M_{X/S} \subset L$ holds. Using the universal properties of fiber product and normalization, we obtain a decomposition
\begin{equation} \label{Eq:FactorizationOfQuotient}
\begin{tikzcd}
A_{S^{\prime \prime}} \arrow[r,"q_1"] \arrow[d] \arrow[rr, bend left,"q"] & X^{\prime} \arrow[r,"q_2"] \arrow[d] & X \arrow[d] \\
S^{\prime \prime} \arrow[r] & S^\prime \arrow[r] & S,
\end{tikzcd}
\end{equation}
where $q_2 \colon X^\prime \to X$ is the \'{e}tale covering associated to $M/K$.

\begin{lem} \label{Lem:ActionOnAbelianScheme}
The morphism $q$ from Diagram~\ref{Eq:FactorizationOfQuotient} is the quotient morphism induced by an action of $G = \Gal(L/K)$ on $A_{S^{\prime \prime}}$.
\end{lem}
\begin{proof}
We first note that $q$ restricted to the generic fiber $X_{L}$ becomes \'etale: By the construction of $A_{S^{\prime \prime}} \to X$, the map $q$ coincides up to isomorphisms $A_{L} \simeq X_{L}$ with the projection $X_{L} \to X_K$, which is \'etale as the base change of the separable extension $L / K$. Fixing an $A_{L}$-equivariant isomorphism $\phi \colon A_{L} \to X_{L}$, we obtain an injection
\[
\Aut(A_{L}/X_K) \longrightarrow \Aut(X_{L}/X_K), \quad \psi \longmapsto \phi \psi \phi^{-1}
\]
which is bijective as both groups are of the same order $\vert G \vert$. Now $\Aut(X_{L}/X_K)$ can be identified with $G$. In particular, $G$ does not only act on $A_{L} = A_K \times_K L$ via its second factor (which - abusing notation - we denote by $\sigma$ for $\sigma  \in G$), but it also acts on $A_{L}$ via
\[
\phi^{-1} \sigma \phi = \phi^{-1} \sigma \phi \sigma^{-1} \sigma = \xi_\sigma \sigma,
\]
where $\xi_\sigma = \phi^{-1} \sigma \phi \sigma^{-1}$ is a translation of $A_{L}$, i.e.\ can be identified with an element of $A_{L}(L)$. One easily sees that the collection $\xi = (\xi_\sigma)_{\sigma \in G}$ satisfies the cocycle condition and that it gives an element in $H^1(G,A_K(L))$ associated to $X_K$. Moreover, every cohomologuous cocycle gives an action on $A_{L}$ in this way, and its quotient is isomorphic to $X_K$.

The group action of $G$ on $A_{L}$ extends to all of $A_{S^{\prime \prime}}$ because the translations $\xi_\sigma \in A_{S^{\prime \prime}}(L) = A_{S^{\prime \prime}}(S^{\prime \prime})$ and the action of $\sigma \in G$ on the second factor of $A_{S^{\prime \prime}} = A \times_S S^{\prime \prime}$ extend. As $q$ and the quotient morphism coincide on the generic fiber of $A_{S^{\prime \prime}}$ - a dense open subset - they have to be the same on all of $A_{S^{\prime \prime}}$, cf.\ \cite{Goertz+Wedhorn}, Corollary~9.9. Therefore, we consider $q$ as a quotient morphism.
\end{proof}

To see how the factorization $q = q_1 \circ q_2$ is related to the Galois group $G$, we recall that the reduction map
\[
	r \colon A(L) = A(S^{\prime \prime}) \longrightarrow A(k^{\prime \prime})
\]
is surjective due to the smoothness of $A$ over $S$, cf. \cite{Liu}, Corollary~6.2.13. Here, $k^{\prime \prime}$ denotes the residue field of the closed point in $S^{\prime \prime}$.

\begin{prop} \label{Prop:TorsorsAsQuotients}
	The morphism $q_1$ from Diagram~(\ref{Eq:FactorizationOfQuotient}) is the quotient morphism induced by restricting the action of $G$ on $A_{S^{\prime \prime}}$ to $H = \{\, \sigma \in G \mid r(\xi_\sigma) = 0 \,\}$, a normal subgroup of $G$. Moreover, the Galois group $G/H$ of $M_{X/S}$ over $K$ is an abelian group.
\end{prop}
\begin{proof}
To see that $H$ is a normal subgroup, we consider it as the kernel of the map
\[
	G \longhookrightarrow \Aut_S(A_{S^{\prime \prime}}) \longrightarrow \Aut_{k^{\prime \prime}}((A_{S^{\prime \prime}})_{k^{\prime \prime}}), \quad \sigma \longmapsto r(\xi_\sigma)
\]
the latter map being the restriction of an automorphism to the closed fiber over the residue field $k^{\prime \prime}$ of $R^\prime$. Note that $\sigma$ restricted to the closed fiber is trivial, as $k = \kappa(R)$ is separably closed. Using the cocycle conditions and the fact that $r$ is invariant under conjugating a point in $A(L)$ by some element of $G$, one shows that $r(\xi_{\sigma \tau \sigma^{-1} \tau^{-1}}) = 0$ for all $\sigma,\tau \in G$, that is, each commutator lies in $H$. Thus, $G/H$ is abelian. The morphism $A_{S^{\prime \prime}} /H \to (A_{S^{\prime \prime}} /H)/(G/H) = A_{S^{\prime \prime}} / G = X$ is \'etale because it arises from an \'etale group scheme acting freely on closed points. Therefore, we have a map $X^\prime \to A_{S^{\prime \prime}} / H$. On the other hand, as $H$ acts trivially on the special fiber, there is no non-trivial factorization of $A_{S^{\prime \prime}} \to A_{S^{\prime \prime}}/H$ into two \'etale morphisms, hence $X^\prime = A_{S^{\prime \prime}}/H$. 
\end{proof}

The description of $X_M$ in terms of cocycles and the reduction map will enable us to use a result of Raynaud \cite{Raynaud1970}. To this end, we additionally need to assume that $k$ is algebraically closed. The argumentation of \cite{Raynaud1970}, Th\'eor\`eme~9.4.1 shows exactly as in the case of elliptic fibrations that for an abelian variety $A_K$ with ordinary good reduction, one has an exact sequence
\[
0 \to H^1(G_{K^{\sep}},A_1(S^{\sep}))_{p^\infty} \to H^1(G_{K^{\sep}},A(K^{\sep}))_{p^\infty} \to H^1(G_{K^{\sep}},A(k))_{p^\infty} \to  0,
\]
where $G_{K^{\sep}}$ is the absolute Galois group of $K$, $A_1(S^{\sep})$ is the kernel of the reduction map $A(K^{\sep}) \to A(k)$ and $H^1(G_{K^{\sep}},A(K^{\sep}))_{p^\infty}$ denotes the subgroup of elements of $p$-power order of $H^1(G_{K^{\sep}},A(K^{\sep}))$. More specifically, $A_1$ is a $p$-divisible group and $S^{\sep}$ the normalization of $S$ in $K^{\sep}$. As we will only work with the $A_M$-torsor $X_M$, we switch to the ground field $M$:

\begin{lem} \label{Lem:TorsorComesFromLeft}
The class of the $A_M$-torsor $X_M$ lies in the image of
\[
H^1(G_{M^{\sep}},A_1(S^{\sep}))_{p^\infty} \longrightarrow H^1(G_{M^{\sep}},A(M^{\sep}))_{p^\infty}.
\]
\end{lem}
\begin{proof}
Let $L / M$ be a minimal Galois extension splitting $X_M$. By \cite{Lang+Tate:PrincipalHomogeneousSpacesOverAbelianVarieties}, Proposition~5, the order $d$ of the torsor $X_M$ in $H^1(G_{M^{\sep}},A(M^{\sep}))$ divides the index $\operatorname{ind}_{X_M}$ of $X_M$, i.e.\ the greatest common divisor of all degrees of separable field extensions splitting $X_M$. Moreover, $d$ and $\operatorname{ind}_{X_M}$ are divisible by the same primes, loc.cit. Now $\operatorname{ind}_{X_M}$ equals $m$, the greatest common divisor of multiplicities of the special fiber of a regular fibration with generic fiber $X_M$, cf.\ \cite{Gabber+Liu+Lorenzini:IndexOfAlgebraicVariety}, Proposition~8.2.(b). But $m$ is a power of $p$ due to Proposition~\ref{Prop:EtaleBaseField} and \ref{Prop:PrimeToEtale}, so $d$ must be a power of $p$ as well and thus an element of $H^1(G_{M^{\sep}},A(M^{\sep}))_{p^\infty}$. Moreover, a cocycle $(\xi_\sigma)_\sigma$ corresponding to $X_{M}$ is mapped to $(r(\xi_\sigma))_\sigma = 0$ under $H^1(G_{M^{\sep}},A(M^{\sep}))_{p^\infty} \to H^1(G_{M^{\sep}},A(k))_{p^\infty}$ according to Proposition~\ref{Prop:TorsorsAsQuotients}. Hence, it has to be in the image of $H^1(G_{M^{\sep}},A_1(S^{\sep}))_{p^\infty}$ in $H^1(G_{M^{\sep}},A(M^{\sep}))_{p^\infty}$.
\end{proof}

\begin{proof}[Proof of Thm~\ref{Thm:TorsionGoodCase} under assumption $(ii)$]
Let $k$ be algebraically closed and $f \colon X \to S$ be an elliptic fibration such that $X_K$ is a torsor under an abelian variety $A_K$ with ordinary good reduction. Using Proposition~\ref{Prop:EtaleFieldDependsGenericFiber} and Proposition~\ref{Prop:InvariancePseudoRational}, it suffices to check the freeness of $H^1(X^\prime, \OO_{X^\prime})$ for any elliptic fibration with generic fiber isomorphic to $X_K$. Therefore, we can assume $f \colon X \to S$ to be a minimal elliptic fibration as constructed in Diagram~\ref{Eq:FactorizationOfQuotient}. Using Lemma~\ref{Lem:TorsorComesFromLeft}, we see that $X_M$ is in the image of $H^1(G_{M^{\sep}},A_1(S^{\sep}))_{p^\infty} \longrightarrow H^1(G_{M^{\sep}},A(M^{\sep}))_{p^\infty}$. Now \cite{Raynaud1970}, Th\'eor\`eme~9.4.1, states that this is equivalent for $X^\prime$ to be cohomologically flat. By the proper base change theorems (e.g. \cite{Hartshorne:AlgebraicGeometry}, Theorem~12.11), this is equivalent to $H^1(X^\prime, \OO_{X^\prime})$ being free.
\end{proof}

\section{Reducing to semi-stable reduction} \label{Sec:AddReduction}

Let $f \colon X \to S$ be an abelian fibration such that the abelian variety $A_K$ under which $X_K$ is a torsor does not have semi-stable reduction. By the theory of semi-stable reduction, there exists a finite extension $K^\prime/K$ such that $A_{K^\prime}$ becomes semi-stable, cf.\ \cite{SGA7_1}, Th\'eor\`eme~3.3.6.4. In particular, as $R$ is complete, there exists a unique minimal such extension, cf.\ \cite{Halle+Nicaise:NeronModels}, Theorem~3.3.6.4. It is Galois and its Galois group $G = \Gal(K^\prime/K)$ has order prime to $p$ whenever the residue characteristic $p$ is bigger than $2g+1$, where $g$ is the dimension of $A_K$, cf.\ \cite{Conrad:SemStabRedForAbelVar}, Theorem~6.8. In the case of elliptic curves, the reduction types after tame base change have been studied in detail in \cite{Dokchitser:TateAlgorihtm}, Theorem~3.


\begin{thm} \label{Thm:AdditiveReduction}
Let $A_K$ be an abelian variety that attains semi-stable reduction over a Galois extension $K^\prime/K$, with Galois group $G$ of order prime to the residue characteristic $p$. Assume that either $A_{K^\prime}$ has good reduction or that $A_K$ is an elliptic curve. If $X_K$ is a torsor under $A_K$ and $X$ a fibration over $S$ with generic fiber isomorphic to $X_K$ and pseudo-rational singularities, then
\[
H^1(X, \OO_X) = H^1(X^\prime, \OO_{X^\prime})^G,
\]
where $X^\prime$ is any proper model of $X_{K^\prime}$ with rational singularities over the integral closure $S^\prime$ of $S$ in $K^\prime$.
\end{thm}
\begin{proof}
Using Proposition~\ref{Prop:InvariancePseudoRational}, it suffices to show the claim for any proper model of $X_K$ with pseudo-rational singularities. We at first assume that $A_{K^\prime}$ has good reduction and take a Galois extension $K^{\prime \prime}$ over $K$ such that $K^\prime \subset K^{\prime \prime}$ and that $X_{K^\prime}$ splits over $K^{\prime \prime}$. Denote by $(\xi_\sigma)_{\sigma \in G_{K^{\prime \prime}}}$ the cocycle in $H^1(G_{K^{\prime \prime}},A_K(K^{\prime \prime}))$ corresponding to the torsor $X_K$, with $G_{K^{\prime \prime}} = \Gal(K^{\prime \prime}/K)$.  As in Lemma~\ref{Lem:ActionOnAbelianScheme}, we let $G_{K^{\prime \prime}}$ act on $A_{S^{\prime \prime}}$ via the automorphism $\xi_\sigma \sigma$ associated to $\sigma \in G_{K^{\prime \prime}}$. Set $X = A_{S^{\prime \prime}} / G_{K^{\prime \prime}}$ and $X^\prime = A_{S^{\prime \prime}} / \Gal(K^{\prime \prime}/K^{\prime})$. Then we saw that - as $A_{K^\prime}$ has good reduction - $X^\prime$ is a regular model of $X_{K^\prime}$. Moreover, the action of $G_{K^{\prime \prime}}$ on $A_{S^{\prime \prime}}$ induces an action of $G = G_{K^{\prime \prime}} / \Gal(K^{\prime \prime}/K^{\prime})$ on $X^\prime$ with $X = X^\prime / G$. By assumption, the order of $G$ is prime to the residue characteristic. Therefore, for a $G$-invariant open affine subset $\Spec(B) \subset X^\prime$, we have $B = B^G \oplus B^\prime$ as $B^G$-modules, and applying \cite{Boutot:SingularitesRationelles} in the equicharacteristic case or \cite{Heitmann+Ma:BigCohenMacaulayAlgebras}, Corollary~4.3 in the positive or mixed characteristic case shows that $\Spec(B^G)$ has pseudo-rational singularities. As $X$ is covered by such spectra, it has pseudo-rational singularities. As $\vert G \vert$ is prime to $p$, it is an invertible element in $H^0(X^\prime, \OO_{X^\prime})$. Hence, multiplying with $\vert G \vert$ induces an isomorphism on $H^j(X^\prime, \OO_{X^\prime})$, which leads to $H^i(G,H^j(X^\prime, \OO_{X^\prime})) = 0$ for $i \ge 1$. Therefore, the spectral sequence $\mathrm{II}_2^{r,s}$ in \cite{Grothendieck:Tohoku}, Th\'eor\`eme~5.2.2 degenerates. The same reasoning leads to the degeneration of $\mathrm{I}_2^{r,s}$ in the same theorem, from which we obtain $H^1(X, \OO_X) = H^1(X^\prime, \OO_{X^\prime})^G$. By Proposition~\ref{Prop:InvariancePseudoRational}, this group is invariant under choosing a pseudo-rational fibration with isomorphic generic fiber.

Now assume that $A_K$ is an elliptic curve. We let $X$ be the minimal regular fibration over $S$ with generic fiber $X_K$ and $Y^\prime \to X \times_S S^\prime$ be the minimal desingularization of the normalization of $X \times_S S^\prime$ over $S^\prime$. Then $G$ acts on $Y^\prime$ as follows: It naturally acts on $S^\prime$, and hence on the product $X \times_S S^\prime$ via its second factor. Now the universal properties of the normalization and the minimal desingularization induce a group action on the normalization and $Y^\prime$ respectively. Writing $Y = Y^\prime/G$, we apply the same reasoning as above to see that $Y$ has pseudo-rational singularities and that $H^1(Y, \OO_Y) = H^1(Y^\prime, \OO_{Y^\prime})^G$ holds. By Proposition~\ref{Prop:InvariancePseudoRational}, we have $H^1(Y^\prime, \OO_{Y^\prime}) = H^1(X^\prime, \OO_{X^\prime})$ and $H^1(Y, \OO_Y) = H^1(X, \OO_X)$ for any fibration $X$ over $S$ with pseudo-rational singularities.
\end{proof}

\begin{rem}
In the setting of Theorem~\ref{Thm:AdditiveReduction}, there exists a projective regular model of $X_K$, cf.\ \cite{Bellardini+Smeets:LogarithmicGoodReduction}, Corollary~1.3.
\end{rem}

\section{Computing the torsion in terms of higher ramification groups}

In this section, we recall a method which enables us to compute the torsion in $H^1(X,\OO_X)$ explicitly by means of \cite{Sen:AutoLocFields}, Theorem~2: We were reduced to compute the cohomology group $H^1(G,R^\prime)$ for a Galois extension $R^\prime / R$ of complete discrete valuation rings, where $G = \Aut(R^\prime / R) = \Gal(K^\prime / K)$. In this setting, it is fruitful to study the higher ramification groups
\[
G_i = \{\, \sigma \in G \mid  \sigma \text{ operates trivially on } R^\prime / \mm^{\prime i+1} \,\}.
\]
They form a descending filtration of $G = G_{-1}$ by normal subgroups that eventually stops, cf.\ \cite{Serre:LocalFields}, Chapter~IV, §1, Proposition~1. Moreover, the quotient $G_0 / G_1$ is cyclic of order prime to $p$ and the groups $G_i / G_{i+1}$ are direct products of elementary abelian $p$-groups, cf.\ \cite{Serre:LocalFields}, Chapter~IV, §2, Corollary~1-3. Note that in our setting, we have $G_0 = G$. Using the degeneration of the Hochschild--Serre spectral sequence, we obtain
\[
	H^1(G,R^\prime) = H^1(G_1, R^\prime)^{G/G_1}.
\]
Note that the extension $R^\prime / R^{\prime G_1}$ is called \emph{wildly ramified}. If $G_1$ is additionally cyclic of order $p^n$ with generator $\sigma$, Sen in \cite{Sen:AutoLocFields} defines the function
\[
i \colon \ZZ \longrightarrow \NN_{\ge 1}, \quad l \longmapsto \nu\left(\frac{\sigma^l(u^\prime)-u^\prime}{u^\prime}\right) = \nu(\sigma^l(u^\prime)-u^\prime)-1,
\]
where $u^\prime \in R^\prime$ is a uniformizer and $\nu$ is the valuation on $K^\prime$ defined by $R^\prime$. It measures for which $j \in \NN$ we have $\sigma^l \in G_j \setminus G_{j+1}$.

\begin{thm} \label{Thm:Sen}
Let $f \colon X \to S$ be an elliptic fibration with $k$ algebraically closed. Denote by $M = M_{X/K}$ the maximal base field and by $E_K$ the Jacobian of $X_K$. Assume that one of the following assumptions holds:
\begin{enumerate}
	\item $E_K$ is an elliptic curve with ordinary good reduction.
	\item $E_K$ is an elliptic curve with multiplicative reduction.
\end{enumerate}
Then the $p$-subgroup $G_1$ of the Galois group $G = \Gal(M/K)$ is cyclic and the torsion part $H^1(G,R^\prime) \subset H^1(X, \OO_X)$ is given by
\[
	H^1(G,R^\prime) \simeq \bigoplus_{l=1}^{p^n-1} R/\mm^{n_l} R,
\]
where $n_l =  \left\lfloor \frac{l + i(l)+(d-1)p^n}{dp^n}\right\rfloor$ with $\vert G_1 \vert = p^n$ and $\vert G/G_1 \vert = d$.
\end{thm}
\begin{proof}
We already know from \cite{Liu+Lorenzini+Raynaud}, Proposition~8.3, and the proof of Theorem~\ref{Thm:TorsionMultCase} that $M/K$ is cyclic in the second case. Thus, it remains to check that $G_1$ is cyclic in the first case of the statement. To this end, we may assume $G = G_1$ by considering $X_{M^{G_1}}$. It is an element of $p$-power order in $H^1(G_{K^{\sep}},A(K^{\sep}))$ as discussed in the proof of Lemma~\ref{Lem:TorsorComesFromLeft}, that is, it is contained in $H^1(G_{K^{\sep}},A(K^{\sep}))_{p^\infty}$. Now $M$ is exactly the splitting field of the image of $X_K$ in $H^1(G_{K^{\sep}},A(k))_{p^\infty}$. By \cite{Raynaud1970}, Th\'eor\`eme~9.4.1, $H^1(G_{K^{\sep}},A(k))_{p^\infty}$ is isomorphic to $\Hom(G_{K^{\sep}},\QQ_p / \ZZ_p)$. Let $\varphi \colon G_{K^{\sep}} \to \QQ_p / \ZZ_p$ be the map corresponding to the image of $X_K$. Then $G = G_{K^{\sep}} / \ker \varphi$, a subgroup of $\QQ_p / \ZZ_p$, hence cyclic.

We are now in the position to apply \cite{Sen:AutoLocFields}, Theorem~2: Let $\sigma$ be a generator of the cyclic $p$-group $G_1$ of order $p^n$ and $i$ be the function defined above in the statement. By loc.cit, we then have
\[
H^1(G_1,R^\prime) \simeq \bigoplus_{l=1}^{p^n-1} R^{\prime G_1}/\mm^{\tilde{n}_l} R^{\prime G_1},
\]
where $\tilde{n}_l = \left\lfloor \frac{l + i(l)}{p^n}\right\rfloor$. On the other hand, because $\vert G/G_1 \vert$ is prime to $p$, we already saw that $H^1(G,R^\prime) = H^1(G_1,R^\prime)^{G/G_1}$. Using \cite{Kock:GaloisStructure}, Lemma~1.5, we obtain
\[
H^1(G,R^\prime) \simeq \bigoplus_{l=1}^{p^n-1} R/\mm^{n_l} R,
\]
where $n_l = 1+ \left\lfloor \frac{\tilde{n}_l-1}{d} \right\rfloor = \left\lfloor \frac{l + i(l)+(d-1)p^n}{dp^n}\right\rfloor$.
\end{proof}


\begin{ex}
Resuming Example~\ref{Ex:ExampleKU}, we observe that $G$ is cyclic and it has its only ramification jump at $m$, i.e.\ we have $G = G_m \supsetneq G_{m+1} = \{1\}$ (cf.\ \cite{Serre:LocalFields}, Chapter~IV, §2, Exercise~5). In particular, writing $m = dp + b$ for $d \ge 0$ and $1 \le b \le p-1$, we deduce from Theorem~\ref{Thm:Sen}
\[
	H^1(G, R^\prime) = \big(k[[t]]/(t^{d+1})\big)^{\oplus d} \oplus \big(k[[t]]/(t^{d})\big)^{\oplus p-b-1}.
\]
\end{ex}

\bibliographystyle{leif3}
\bibliography{Literatur}

\end{document}